\documentclass[times]{article}
\usepackage{vmargin}
\usepackage{tikz}
\usetikzlibrary{shapes,arrows}
\usepackage{amssymb}
\usepackage{amsmath}
\usepackage{amsthm}
\usepackage{amsfonts}
\usepackage[utf8]{inputenc}
\usepackage{stmaryrd}
\usepackage[sans]{dsfont}

\newcommand{\trans}{\mathsf{T}}
\newcommand{\e} {\varepsilon}
\newcommand{\Z} {\mathbb{Z}}
\newcommand{\N} {\mathbb{N}}

\newcommand{\R} {\mathbb{R}}

\newcommand{\di}[1]{\operatorname{d}\!#1}

\newtheorem{theorem}{Theorem}
\newtheorem{lemma}{Lemma}

\newcommand{\vp}{\varphi}
\newcommand{\p}{\mathcal{P}}

\newcommand{\abs}[1]{|#1|_{\text{vec}}}

\begin{document}


\title{Asymptotic Stability and Smooth Lyapunov Functions for a Class of Abstract Dynamical Systems
}


\author{Michael Sch\"onlein\footnote
{Institute for Mathematics, University of W\"urzburg, Emil-Fischer
    Stra\ss e 40, 97074 W\"urzburg, Germany, {\tt schoenlein@mathematik.uni-wuerzburg.de}}}


\maketitle

\begin{abstract}
This paper deals with asymptotic stability of a class of dynamical systems in terms of smooth Lyapunov pairs. We point out that well known converse Lyapunov results for differential inclusions cannot be applied to this class of dynamical systems. Following an abstract approach we put an assumption on the trajectories of the dynamical systems which demands for any trajectory the existence of a neighboring trajectory such that their difference grows linearly in time and distance of the starting points. Under this assumption, we prove the existence of a $C^\infty$-smooth Lyapunov pair. We also show that this assumption is satisfied by differential inclusions defined by Lipschitz continuous set-valued maps taking nonempty, compact and convex values.
\end{abstract}


\section{Introduction}

Inspired by the work of Lyapunov starting in the 1950s a lot of effort has been spent on the stability analysis of dynamical systems in terms of Lyapunov functions. Beginning with ordinary differential equations defined by a continuous function, the results have been extended to differential inclusions
\begin{align}\label{DI}
\dot x(t) \in F(x(t)), \qquad x(t) \in \R^n,\,t\geq 0 ,\, x(0)=x_0 \in \R^n,
\end{align}
where $F\colon \R^n \rightsquigarrow \R^n$ is a set-valued map satisfying $0 \in F(0)$. A comprehensive exploration of the connection between stability of differential equations and differential inclusions and Lyapunov functions can be found in \cite{bacciotti}. For general references to set-valued maps and differential inclusions the interested reader is referred to \cite{aubin-frankowska} and \cite{aubin-cellina,smirnov}, respectively. 

A solution to \eqref{DI} is an absolutely continuous function $x: \R_+ \to \R^n$ with $x(0)=x_0$ such that \eqref{DI} is satisfied almost everywhere. Following \cite[Proposition~2.2]{clarke98} the equilibrium $x=0$ of differential inclusion \eqref{DI} is called \emph{strongly asymptotically stable} if each solution can be extended to $[0,\infty)$, for any $\e>0$ there is a $\delta>0$ such that
any solution $x(\cdot)$ with $\|x(0)\|<\delta$ satisfies $\|x(t)\| < \e$ for all $t\geq 0$, and for each individual solution $x(\cdot)$, one has $\lim_{t\to \infty} x(t)=0$. 

The analysis of robust stability has been an active field in the dynamical systems literature. In the wake of this, the investigation of converse Lyapunov theorems and, in particular, the construction of smooth Lyapunov functions is of vital interest, cf. \cite{clarke98,siconolfi2012,teel2000smooth}.

Clarke, Ledyaev and Stern \cite{clarke98} (see also \cite{teel2000smooth}) have shown that, provided $F(x)$ is nonempty, compact and convex for every $x \in \R^n$ and the set-valued map $F$ is upper semicontinuous, i.e. for any $x \in \R^n$ and any $\e>0$ there is a $\delta>0$ such that
$F(y)\subset F(x) + \e B(0,1)$ for all $y \in x + \delta B(0,1)$, where $B(0,1)$ denotes the unit open ball in $\R^n$,
the differential inclusion \eqref{DI} is strongly asymptotically stable if and only if there is a $C^\infty$-smooth and positive definite pair of functions $(V,W)$ such that $V$ is proper and
\begin{align}\label{ineq:V-decrease}
\max_{v \in F(x)} \langle \nabla V(x), v \rangle \leq - W(x)\qquad \text{ for all }\, x \in \R^n\setminus \{0\}.
\end{align}
A different proof of this converse Lyapunov theorem following a metric approach using weak KAM theory has been obtained by Siconolfi and Terrone, cf. \cite{siconolfi2012}. Related results for retarded functional equations and difference inclusions can be found in  \cite{karafyllis06} and \cite{kelletteel04}, respectively.

Originating from stochastic systems, such as multiclass queueing networks and semimartingale reflected Brownian motions \cite{dai,DP94}, there is a class of dynamical systems that does not immediately fall into framework mentioned above. More specifically, the analysis of recurrence behavior of the stochastic processes corresponding to multiclass queueing networks or semimartingale reflected Brownian motions, based on the remarkable insights of \cite{dai,DP94,rybkostolyar}, can be reduced to the stability analysis of a related deterministic system, called fluid network and linear Skorokhod problem, respectively. Both models are obtained by taking limits of scaled versions of the stochastic processes.  In \cite{DP94,questa-2012} it is outlined that a wide class of linear Skorokhod problems and fluid networks can be defined by differential inclusions in a natural way. An essential part in \cite{DP94} is to describe the linear Skorokhod problem in terms of a differential inclusion and construct a $C^1$-Lyapunov function. The paper by Dupius and Williams \cite{DP94} was published a few years before the above mentioned paper on smooth Lyapunov functions by Clarke, Ledyaev and Stern \cite{clarke98}.


The differential inclusions considered in \cite{clarke98} and \cite{DP94} are both defined by an upper semicontinuous set-valued map $F$ with nonempty, compact and convex values. The techniques used to construct a smooth Lyapunov function have in common that the set-valued map $F$ is embedded into a local Lipschitz set-valued map $F_L$, which  keeps the property of asymptotic stability. Whereas the procedure in \cite{DP94} uses explicitly the properties of the set-valued map describing the evolution of the linear Skorokhod problem, the embedding technique in \cite{clarke98} is applicable to any upper semicontinuous set-valued map taking nonempty, compact and convex values. The essential feature of the set-valued map $F_L$ being local Lipschitz continuous is that it provides a Lyapunov function which is locally Lipschitz continuous and this property can be carried over to conclude a locally Lipschitz continuous Lyapunov function for the original differential inclusion. Moreover, the local Lipschitz continuity of the set-valued map $F_L$ facilitates to establish that the convolution of the local Lipschitz continuous Lyapunov function and a $C^\infty$-smooth mollifier satisfies locally the decrease condition \eqref{ineq:V-decrease}. The construction is completed by using a locally finite open covering of $\R^n$ and a smooth partition of unity subordinate to it.

We show that, in general the set-valued map defining the differential inclusion describing the evolution of a fluid network is not upper-semicontinuous. Thus, although the the zero solution may be strongly asymptotically stable the existence of a $C^\infty$-smooth Lyapunov pair cannot be concluded by the results on differential inclusions mentioned above. In this paper we follow an abstract point of view, starting with Zubov \cite{zubov1964}, understanding dynamical systems as abstract mathematical objects with certain properties. This has been further explored by Hale, Infante, Slemrod and Walker, cf. \cite{hale69,haleinfante67,slemrod70,walker78}. In the literature there are several terms used, for instance, generalized dynamical system, $C^0$-semigroup, (semi)flow, process or abstract dynamical system, cf. \cite{walker1980} and the references therein. The class of abstract dynamical systems considered in this paper is defined by the characteristic properties of fluid networks. The trajectories of fluid networks evolve in the positive orthant. In order to get a $C^\infty$-smooth Lyapunov function on the positive orthant we use an extension of a Lyapunov function candidate to $\R^n$ by taking absolute values component-by-component. As this defines a continuous map the extended Lyapunov function is continuous as well. We note that, as the solutions to linear Skorokhod problems also stay within the positive orthant, Dupuis and Williams \cite{DP94} solved the boundary problem by shifting the orthant by some positive constant. Further, we note that the class of abstract dynamical systems under consideration may in general not be defined by a differential inclusion. As a consequence the constructions of a local Lipschitz continuous Lyapunov function in \cite{clarke98,DP94,siconolfi2012}, which are based on the right-hand side of the differential inclusion, are not applicable in the present setting. It turns out that the essential ingredient to obtain a local Lipschitz continuous Lyapunov function is an estimate on the evolution of the difference of trajectories. For this reason, we have to make an assumption on the trajectories of the abstract dynamical system (see assumption~(A) in Theorem~\ref{thm:smooth-conv-lyap}). Considering the assumption from the differential inclusions perspective we show that it is automatically satisfied for every differential inclusion with Lipschitz continuous right-hand side.

The paper is organized as follows. In Section~2 we state relevant notation and terminology that is used throughout the paper. Section~3 introduces the class of abstract dynamical systems that is considered and the main result is presented. In Section~4 we outline that the class of abstract dynamical systems is motivated by the analysis of fluid networks. We also show that the classical results on smooth Lyapunov functions do not apply to the class of abstract dynamical systems discussed in this paper. In Section~5 we examine the relation of the assumption posed on the trajectories in the light of differential inclusions. Finally, Section~6 is devoted to the proof of the main result.

\section{Notation and terminology}

A function $f \colon \R^n \to \R$ is called \emph{proper} if the sublevel sets $\{ x \in \R^n | f(x) \leq c\}$ are bounded for all $c>0$. For $r>0$ and $x\in \R^n$ let $B(x,r):=\{ y \in \R^n | \|x-y\| \leq r\}$. A function $k \in C^{\infty}(\R^n,\R_+)$ is called a \emph{mollifier} if $\mbox{supp } k=B(0,1)$ and 
\begin{align*}
 \int_{\R^n} k(x)\,\di{x} =1.
\end{align*}
Furthermore, the support of a mollifier can be scaled in the following way. For $r>0$ consider $ k_{r}(x):= \tfrac{1}{r^{n}}\, k(r^{-1}x)$. Then, it follows $k_{r} \in C^{\infty}(\R^n,\R_+)$, $\text{ supp } k_{r}=B(0,r)$, and
\begin{align*}
\int_{\R^n} k_{r}(x) \,\di{x}=1.
\end{align*}
Moreover, to consider the convolution of a function $f \in C(\R^n,\R)$ and a mollifier $k_{r}$, let $U$ be an open subset of $\R^n$ and $U_{r} = \{ x \in U\, |\, d(x,\partial U) >r\}$. Then, the \emph{convolution}, denoted by $f_r: U_r \rightarrow \R$, is defined by
\begin{equation*}
x \mapsto f_{r}(x) := f \ast k_{r} \, (x) = \int_{B(0,r)} f(x-y)\, k_{r}(y)\, \di{y}.
\end{equation*}
By standard convolution results it follows $f_{r} \in C^{\infty}(U_{r},\R_+)$, see for instance \cite[Theorem 6 Appendix C.4]{evans-pde}. Furthermore, if $f$ is continuous in $U$, it holds $f_r \rightarrow f$ uniformly on compact subsets (u.o.c.) of $U$ as $r\rightarrow 0$. 
The \emph{Dini subderivative} of a function $f:U \rightarrow \R$ at $x \in U$ in the direction $v\in \R^n$ is defined by
\begin{align*}
 Df(x;v):= \liminf_{\e \rightarrow 0, v'\rightarrow v} \frac{f(x+\e v')-f(x)}{\e}.
\end{align*}
Let $T(x,\R_+^n)$ denote the \emph{contingent cone} to $\R^n_+$ at $x$ defined by
\begin{align*}
T\big(x,\R_+^n\big)= \left\{ \, v\in \R^n \,| \, \liminf_{\varepsilon \to 0} \; \frac{ d(x+\varepsilon v, \R_+^n) }{\e} \, =0\right\},
\end{align*}
with $d(x,K) = \inf\{ \|x-y\|\, | y \in K \}$.

\section{Statement of the main result}

We start this section by recalling an abstract definition of a dynamical system from \cite{walker1980}. A \emph{dynamical system} defined on a metric space $X$ is a continuous mapping $u\colon \R_+ \times X \to X$ such that $u(0,x)=x$ and
$$  u(t,u(s,x))=u(t+s,x) \quad \text{ for all } t,s \in \R_+,\,x \in X.$$
Recall that $x_* \in X$ is an equilibrium if $u(t,x_*)= x_*$ for all $t\geq 0$ and $x_*$ is said to be \emph{stable} if for every $\e>0$ there is a $\delta >0$ such that $d(x,x_*) < \delta$ implies that $ d\left(u(t,x),u(t,x_*) \right)< \e$ for all $t \geq0$. If in addition, there is a $M>0$ so that $d(x,x_*)< M$ implies that $\lim_{t \to \infty} d\left(u(t,x),u(t,x_*) \right) =0$, then $\vp_*$ is called \emph{asymptotically stable}.

Here we consider the metric space $\p \subset C(\R_+,\R_+^n)$ defined by the following properties: 
\begin{enumerate}
\item[$(a)$] There is a $L>0$ such that 
\begin{equation*}
\|\varphi(t)-\varphi(s)\| \leq L \, |t-s| \qquad \text{ for all } \vp \in \p, \, t,s \in \R_+.
\end{equation*}
\item[$(b)$] Scaling invariance: $\tfrac{1}{r}\,\varphi(r\,\cdot)  \in \p$ for all $\varphi \in \p,r>0$.
\item[$(c)$] Shift invariance: $\varphi(\cdot+t) \in \p$ for all $\varphi \in \p,t\geq0$.
\item[$(d)$] If a sequence $(\varphi_n)_{n \in \N}$ in $\p$ converges to $\varphi_*$ uniformly on compact sets, then $\varphi_* \in \p$.
\item[$(e)$] Concatenation property: For all $\varphi_1,\varphi_2 \in \p$ with $\varphi_1(t^*)=\varphi_2(0)$ for some $t^* \geq 0$ it holds $\varphi_1\diamond_{t^*}\varphi_2 \in \p$, where
\begin{equation*}
\varphi_1\diamond_{t^*}\varphi_2(t):= 
\begin{cases} 
\varphi_1(t) &\quad t \leq t^*, \\
\varphi_2(t-t^*) &\quad t \geq t^*.
\end{cases}
\end{equation*}
\item[$(f)$] There is a $T>0$ such that the set-valued map $ P: \R^n_+  \rightsquigarrow \p$ defined by
\begin{equation*}
 P(x)= \{ \varphi:[0,T] \to \R^n_+  \, | \, \varphi \in \p , \, \varphi(0)=x \}
\end{equation*}
is lower semicontinuous, i.e. for any $\vp \in P(x)$ and for any sequence of elements $(x_n)_{n \in \N} \in \R_+^n$ converging to $x$, there exists a sequence $(\vp_n)_{n \in \N}$ with $\vp_n \in P(x_n)$ converging to $\vp$ uniformly on compact sets.
\end{enumerate}

By condition (a) the functions $\vp \in \p$ are Lipschitz continuous with respect to a global Lipschitz constant. Condition (c) is in one-to-one correspondence to time-invariance of differential equations/inclusions. Condition (d) expresses that the set $\p$ is closed in the topology of  uniform convergence on compact sets. The concatenation property (e) and condition (f) imply that for $x\in \R^n_+$ the set of functions in $\p$ starting in $x$ depend lower semicontinuously on $x$, cf. \cite{questa-2012}.

In order to define the class of dynamical systems which will be considered in this paper, we equip the set $\p$ with the metric
\begin{align*}
d(\varphi_1,\varphi_2) := \max_{N \in \N}  \,  2^{-N}\frac{\|\vp_1 - \vp_2\|_N}{1 +\|\vp_1-\vp_2\|_N},  
\end{align*}
where $\| \vp\|_N:= \sup_{t \in [0,N]} \|\vp(t)\|$ so that convergence of functions is equivalent to uniform convergence of the corresponding restrictions on each compact subset of $\R_+$, cf. \cite{rudin1991functional}. Furthermore, we define the shift operator 
\begin{align*}
S(t) \colon C(\R_+,\R_+^n) \to C(\R_+, \R_+^n), \quad S(t) \varphi(\cdot) = \varphi(\cdot +t).
\end{align*}
In this paper, we consider the dynamical system on the metric space $\p \subset C(\R_+,\R^n_+)$ defined by the mapping 
\begin{align*}
 u \colon \R_+ \times \p\to \p, \qquad u(t,\vp)=S(t) \varphi(\cdot)=\vp(\cdot + t).
\end{align*}

Throughout the paper we call a function $\vp \in \p$ a \emph{trajectory} of $\p$. The zero trajectory $\varphi_*\equiv 0$ is the unique fixed point of the shift operator $S(t)$ defined on $\p$ and thus, $\varphi_*\equiv 0$ is the only equilibrium of the dynamical system.
The scope of this paper is to characterize asymptotic stability of the dynamical system $u$ defined on $\p$ in terms of the existence of a smooth Lyapunov function.

A pair $(V,W)$ of positive definite functions on $\R_+^n$ is called a \emph{Lyapunov pair} for the dynamical system $u$ defined on $\p$ if $V \colon \R_+^n \rightarrow \R_+$ is proper and for any $\vp \in \p$,
\begin{align}\label{fLF2}
V(\varphi(t)) - V(\varphi(s)) &\leq - \int_{s}^{t} \,W(\varphi(r))\,\di{r} \qquad \text{ for all }0\leq s \leq t \in \R_+.
\end{align}
The pair $(V,W)$ is called a \emph{$C^\infty$-smooth Lyapunov pair} if the functions $V$ and $W$ are $C^\infty$-smooth.
To formulate the decrease condition in a differential form let
\begin{align}\label{V-smooth-decrease}
\dot V(\vp(t)) := \lim_{h\rightarrow 0} \frac{V(\vp(t+h)) - V(\vp(t))}{h}.
\end{align}
The main result of the paper is the following.

\begin{theorem}\label{thm:smooth-conv-lyap}
Suppose the dynamical system $u$ defined on $\p$ satisfies:
\begin{itemize}
\item[\emph{(A)}] For any $\varphi \in \p$, $\e>0$, and $T>0$ there is a continuous function $c\colon [0,T] \to \R_+$ such that $\lim_{t \to 0} \tfrac{c(t)}{t}$ exists and is positive and for any $y \in \R^n$ with $\vp(0)-y \in B(\vp(0),\e) \cap \R_+^n$ there is a trajectory $\psi \in \p$ with $\psi(0)=\varphi(0)-y$ satisfying
\begin{align*}
\| \,\varphi(t) - y - \psi(t)\, \| &\leq   \|y\|\, c(t) \qquad \text{ for all } \, \, t \in [0,T]. 
\end{align*}
\end{itemize}
Then, the dynamical system $u$ defined on $\p$ is asymptotically stable if and only if there is a $C^{\infty}$-smooth Lyapunov pair $(V,W)$ such that for every $\varphi \in \p$ it holds
\begin{align}\label{DC}
\dot V(\varphi(t)) \leq  - W(\varphi(t)) \qquad \text{ for all }\, t \geq 0.
\end{align}
\end{theorem}

\medskip
We note that estimates similar to (A) for trajectories of differential inclusions with state constraints were derived by Bressan, Facchi, Bettiol, Vinter and Rampazzo, cf. \cite{bressan2011} and the references therein. 

\section{Motivation and application of the main result}

For the purpose of this paper we give a very brief introduction to fluid networks. For a comprehensive description of multiclass queueing networks and fluid networks we refer to \cite{bramsonLN,chen,dai}. A fluid networks consists of $J\in \N$ stations serving $n\in \N$ different types of fluids with $J \leq n$ and each fluid type is served exclusively at a predefined station. This assignment defines the constituency matrix $ C\in \R^{J \times K}$ with $c_{jk} := 1$ if fluid type $k \in \{1,...,n\}$ is served at station $j\in\{1,...,J\}$ and $c_{jk}=0$ otherwise. The exogenous inflow rate of fluid type $k$ is denoted by $\alpha_k$ and $\alpha=(\alpha_1,...,\alpha_n)^\trans \in \R_+^n:=\{ x \in \R^n :  x_i\geq 0, \,  i =1,...,n\}$ is called the exogenous inflow rate. Likewise, $\mu_k \in \R_+$ denotes the potential outflow rate of type $k$ fluids and $\mu=(\mu_1,...,\mu_n)^\trans \in \R_+^n$. The substochastic matrix $P \in [0,1]^{n \times n}$ describes the transitions in the network, where it is assumed that the spectral radius of the matrix $P$ is strictly less than one, i.e. $1> \max\{|\lambda| \,|\, \exists\, x\neq 0 \,:\, Px=\lambda x\}$. The initial fluid level and the fluid level at time $t$  of the network are denoted by $x(0)$ and $x(t)$, respectively. We note that, as $x(t)$ described the deterministic analog of the queue length of the multiclass queueing network, the fluid level process $x(\cdot)$ evolves only in the positive orthant $\R_+^n$. 

The discipline, for instance First-In-First-Out (FIFO), determines the rule under which the individual stations of the fluid network are serving the different fluid types. More formally, the discipline specifies the allocation rate, denoted by $u =(u_1,...,u_n)^\trans\in \R_+^n$, where $u_k$ denotes the current amount of time that a predefined station allocates to serve fluids of type $k$. For the class of general work-conserving fluid networks, given $x \in \R_+^n$, the set of admissible allocation rates is
\begin{align}\label{U1}
 U(x)= \big\{ u \in \R^n \,\big| \, u \geq 0 ,\quad
e-Cu\geq0,\quad
(Cx)^\trans\cdot(e-Cu) = 0\big\},
\end{align}
where $e=(1,...,1)^\trans \in \R^J$ and the inequalities have to be understood component-by-component. The evolution of the fluid network can then be described by the following differential inclusion
\begin{align}\label{D1}
\dot x & \in G\big(x\big) :=\Big\{ \alpha -(I-P^\trans) M u \,\big|\, u \in U(x) \Big\} \, \cap \, T\big(x,\R^n_+\big),
\end{align}
where $M=\mbox{diag}(\mu)$ and $T(x,\R_+^n)$ denotes the contingent cone to $\R^n_+$ at $x$. The intersection with the contingent cone to the positive orthant is to ensure the nonnegativity of the solutions. Tackling a simple example we show

\begin{lemma}\label{lem:counter-ex}
The set-valued map $G$ defined in \eqref{D1} is not upper semicontinuous in general.
\end{lemma}

\begin{proof}
To show the claim we consider a single station fluid network  serving one type of fluid. That is, for  $\alpha >0$, $\mu=\alpha +1$, and $P=0$ the differential inclusion \eqref{D1} is defined by set-valued map
\begin{equation}\label{def:F-example}
\begin{split}
 G(x) &=\left\{ 
\alpha - (\alpha +1) \, u
\quad \big| \quad
0\leq  u \leq  1, \quad x(1-u)= 0   
\right\} \cap T(x,\R_+)  \\
&=\left\{ 
\alpha - (\alpha +1) \, u
\quad \big| \quad
0\leq  u \leq  1, 
  u = 1  \text{ if } x>0  \text { and } u= \tfrac{\alpha}{\alpha + \delta}  \text{ else } 
 \right\} .
\end{split}
\end{equation}
To conclude that $G$ is not upper semicontinuous let $x=0$ and consider the sequence $(x_k)_{k\in \mathbb{N}}$ with $x_k= \tfrac{1}{k}$. Then, for each $k \in \N$ we have $G(x_k)= -1$. Let $\operatorname{graph}(G) :=\{ \, (x,y) \in \mathbb{R}^n_+ \times \R^n \,\big|\, y \in G(x) \, \}$ denote the graph of $G$ and consider the sequence $\big(x_k,-1\big)_{k \in \N}$ on the graph of $G$ 
which converges to $\big(0,-1\big) \not \in \operatorname{graph}(G)$. Hence, the graph is not closed and by Proposition~2 in \cite[Section~1.1]{aubin-cellina}  the set-valued map $G$ is not upper semicontinuous. 
\end{proof}

We note that the differential inclusion defined by \eqref{def:F-example} is strongly asymptotically stable. However, as a consequence of Lemma~\ref{lem:counter-ex}, the existence of a $C^\infty$-smooth Lyapunov pair cannot be concluded by the results on differential inclusions mentioned above.

\medskip
We use the main result Theorem~\ref{thm:smooth-conv-lyap} to conclude that the differential inclusion \eqref{def:F-example} admits a $C^\infty$-smooth Lyapunov pair. It is well known that a single station general work-conserving fluid network serving only one fluid type satisfies the properties (a)-(f); cf. \cite{chen,questa-2012,stolyar95}. To conclude the existence of a $C^\infty$-smooth Lyapunov function we will apply the main result of the paper by interpreting the zero solution $\vp_*\equiv 0$ of the differential inclusion \eqref{D1} as the fixed point of the shift operator defined on the set of solutions to \eqref{D1}.

\begin{theorem}\label{lem:ex-hypo-yes}
The differential inclusion \eqref{def:F-example} admits a $C^\infty$-smooth Lyapunov pair, i.e. there is a $C^\infty$-smooth and positive definite pair $(V,W)$ such that $V$ is proper and
\begin{align}\label{ineq:V-decrease}
\max_{v \in G(x)} \langle \nabla V(x), v \rangle \leq - W(x)\qquad \text{ for all }\, x \in (0,\infty).
\end{align}
\end{theorem}

\begin{proof}
To show the existence of a $C^\infty$-smooth Lyapunov pair we verify that the set of solutions $\mathcal{S}_G$ to the differential inclusion \eqref{def:F-example} satisfies the assumption~(A). 

Let $\vp \in \mathcal{S}_G$, $\e >0$ and $T>0$ be fixed. In a first step, we treat the case that $\vp(0)>0$. Then, we have 
\begin{align*}
\vp(t) =
\begin{cases} 
   \vp(0) -  t        & \text{ if } t \leq \vp(0)\\
 0 &\text{ else}.
\end{cases}
\end{align*}
In the case $y=\vp(0)$ the only solution $\psi$ to the differential inclusion starting in $\psi(0)=\vp(0) - y=0$ is the zero solution and we obtain $| \vp(t) - y - \psi(t)| =|t|$ for all $ t \leq \vp(0)$ and 
$| \vp(t) - y - \psi(t)| = |y|$ otherwise. Hence, one has
\begin{align*}
| \vp(t) - y - \psi(t)| \leq \frac{1}{\vp(0)}\, |y|\, t \qquad \text{ for all } t\geq 0.
\end{align*}
If $y\neq \vp(0)$ we consider the solution $\psi$ to the differential inclusion starting in $\psi(0)=\vp(0) - y$ given by
\begin{align*}
\psi(t) =
\begin{cases} 
   \vp(0) -y -  t        & \text{ if } t \leq \vp(0)-y\\
 0 &\text{ else}.
\end{cases}
\end{align*}
On one hand, if $0<\vp(0) -y<\vp(0)$ it follows $| \vp(t) - y - \psi(t)| =0$ for all $ t \in [0, \vp(0)-y]$. Also, for all $ t \in [\vp(0)-y, \vp(0)]$ we have $| \vp(t) - y - \psi(t)| =| \vp(0)-y-  t|$ and for all $t\geq \vp(0)$ one has $| \vp(t) - y - \psi(t)| =|y|$. 
On other hand, if $0<\vp(0)\leq \vp(0) -y$ we have $| \vp(t) - y - \psi(t)| =0$ for all $ t \in [0, \vp(0)]$. Further, for all $ t \in [\vp(0),\vp(0)-y]$ one has $| \vp(t) - y - \psi(t)| =| \vp(0)-y- t|$ and for all $t\geq \vp(0)-y$ it follows $| \vp(t) - y - \psi(t)| =|y|$. Consequently, in either case one obtains
\begin{align*}
| \vp(t) - y - \psi(t)| \leq \frac{1}{\vp(0)}\, |y|\, t \qquad \text{ for all } t\geq 0.
\end{align*}
Finally, we consider the case $\vp(0)=0$. Then, we have $\vp \equiv 0$. For $y \in (-\e, 0]$ a solution $ \psi$ of the differential inclusion with $\psi(0)=-y$ is
\begin{align*}
\psi(t) =
\begin{cases} 
   -y -  t        & \text{ if } t \leq -y\\
 0 &\text{ else}.
\end{cases}
\end{align*}
Therefore, $| \vp(t) - y - \psi(t)| =t$ for all $ t \in [0, -y]$ and $| \vp(t) - y - \psi(t)| =|y|$ for all $ t \geq -y$ and one has 
\begin{align*}
 | \vp(t) - y - \psi(t)|\leq \log(t+\operatorname{e})\, |y| \qquad \text{ for all } t\geq 0.
\end{align*}
Thus, there is a $C^\infty$-smooth Lyapunov pair $(V.W)$ such that for every solution $\vp \in \mathcal{S}_G$ one has
\begin{align*}
\dot V( \vp(t)) \leq - W(\vp(t))  \qquad \text{ for all } t\geq 0.
\end{align*}
Consequently, the pair $(V,W)$ satisfies
\begin{align*}
\max_{v \in G(x)} \langle V(x),v\rangle  \leq - W(x)  \qquad \text{ for all } x \in (0,\infty).
\end{align*}
This shows the assertion.
\end{proof}

\section{Relation to differential inclusions}

Due to the fact that Clarke, Ledyaev and Stern \cite{clarke98} as well as Dupius and Williams \cite{DP94} embed the set-valued map defining the differential inclusion into a Lipschitz one, we consider the differential inclusion
\begin{align}\label{eq:Li-di}
 \dot x(t) \in F(x(t)),
\end{align}
where $F:\mathbb{R}^n \rightsquigarrow \mathbb{R}^n$ is Lipschitz continuous, i.e. there is a constant $L>0$ such that
\begin{align*}
 F(x) \subset F(y) +  L\,\|x-y\| \,B(0,1) \qquad \text{ for all }\, x,y\in \R^n,
\end{align*}
and $F(x)$ is nonempty, compact and convex for every $x \in \R^n$. Let  $\mathcal{S}_F(x)$ denote the set of solutions $\varphi\colon \R_+ \to \R^n$ to \eqref{eq:Li-di} with $\varphi(0)=x$. Let $\mathcal{S}_F:= \{ \mathcal{S}_F(x)\, |\, x \in \mathbb{R}^n \}$.
Obviously, the mapping 
\begin{align}\label{def:di-ads}
u: \mathbb{R}_+ \times \mathcal{S}_F \to \mathcal{S}_F ,  \quad u(t,\varphi) := \varphi(t+\cdot)
\end{align}
defines a dynamical system on $\mathcal{S}_F$. 

Next we show that condition~(A) is a natural assumption which is satisfied by a wide class of dynamical systems. To this end, we investigate assumption~(A) from differential inclusions perspective and show that it is automatically fulfilled for differential inclusions defined by a Lipschitz continuous set-valued map.

\begin{theorem}\label{lem:DI-H}
Let $F$ be a Lipschitz continuous set-valued map taking nonempty, compact and convex values with $0 \in F(0)$. Then, the dynamical system $u$ defined on $\mathcal{S}_F$ by \eqref{def:di-ads} satisfies assumption~\emph{(A)}.
\end{theorem}

\begin{proof}
Let $\varphi \in \mathcal{S}_F$, $\varepsilon>0$ and $T>0$. We define $c(t):=\operatorname{e}^{Lt}-1$. Then, for $y \in \mathbb{R}^n$ the function $\varphi_y(\cdot):= \varphi(\cdot)-y$ is absolutely continuous with $\varphi_y(0)=\varphi(0)-y$. Further, as $F$ is Lipschitz it holds $F(\varphi_y(t)) \subset F(\varphi(t)) + L \,\|y\|\,  B(0,1)$ and we have
\begin{align*}
d(\dot \varphi_y(t),F(\varphi_y(t)) = d(\dot \varphi(t),F(\varphi_y(t))\leq L\, \|y\|.
\end{align*}
Thus, by Filippov's Theorem \cite[Theorem~1 in Chapter~2, Section~4]{aubin-cellina} there is a solution $\psi(\cdot)$  to \eqref{eq:Li-di} defined on the interval $[0,T]$ with $\psi(0)=\varphi_y(0)=\varphi(0)-y$ satisfying
\begin{align*}
\| \varphi_y(t)-  \psi(t) \| = \| \varphi(t) -y -  \psi(t) \|  \leq  \|y\| \left( \operatorname{e}^{Lt} -1\right)\quad \text{ for all } t \in [0,T].
\end{align*}
This shows the assertion.
\end{proof}


\section{Proof of the main result} \label{sec:proof}
%
Before proving Theorem~\ref{thm:smooth-conv-lyap} we present a useful characterization of asymptotic stability for the class of dynamical systems under consideration which is based on the scaling property (b) and the shift property (c).

\begin{lemma}\label{prop:stab}
The dynamical system $u$ defined on $\p$ is asymptotically stable if and only if there is a $\tau >0$ such that $u(\|\varphi(0)\|\tau,\vp) =  0$ for all $\varphi \in \p$.
\end{lemma}
\begin{proof}
Suppose there is a $\tau < \infty$ such that $u(\|\varphi(0)\|\tau,\vp) =  0$ for all $\varphi \in \p$. To conclude stability let $\e>0$ and $\delta:=\tfrac{\e}{\lceil L\tau \rceil}$, where for $x\in \R$ let $\lceil x \rceil := \min\{k\in \Z\,| \, k\geq x\}$. Let $\varphi \in \p$ with $d(0,\varphi) < \delta$. By assumption, using the scaling and shift property it holds 
\begin{align}\label{1}
\varphi(s+t)= 0  \quad \text{ for all } t \geq \tau\,\|\vp(s)\|.
\end{align}
Together with the Lipschitz continuity of $\varphi$ for every $t \in [0,\tau\|\vp(s)\|]$ we have
\begin{align}\label{2}
 \|\varphi(s+t)\| =
\| \, \varphi(s+t)- \varphi\big(s+\tau \|\varphi(s)\|\big) \,\| \leq L\, \big| t - \tau \|\varphi(s)\| \, \big| 
\leq  L\tau \|\varphi(s)\|.
\end{align}
By \eqref{1} we conclude that \eqref{2} holds for all $t \geq 0$. Therefore,
\begin{align*}
 \|\varphi(\cdot+t)\|_N=  
\sup_{s \in [0,N]} \| \varphi(s+t)\| 
\leq \lceil L\tau \rceil \|\varphi\|_N =  \| \lceil L\tau \rceil \varphi\|_N .
\end{align*}
In turn, by the triangular inequality, we have
\begin{align*}
d\big(0, u(t,\varphi)\big) \leq \max_{N\in \N} \frac{1}{2^N} \frac{\| \lceil L\tau \rceil \varphi\|_N }{1+\| \lceil L\tau \rceil \varphi\|_N } = d\big(0, \lceil L\tau \rceil \varphi \big) 
\leq \lceil L\tau \rceil d\big(0, \varphi \big) < \e.
\end{align*}
By assumption, it holds $\vp\big(\|\vp(0)\|\tau +t\big) =0$ for all $t\geq 0$. This in turn implies $\lim_{t \rightarrow \infty} d\big(0, u(t,\varphi)\big) =0$ and we have attractivity.

Conversely, let $\vp_*\equiv 0$ be asymptotically stable. Due to the scaling property it suffices to consider trajectories $\vp$ with $\|\vp(0)\|=1$. Then, as 
\[
 \lim_{t \to \infty} d(0,u(t,\vp)) = \lim_{t\to \infty} d(0,\vp(t+\cdot)) =0 
\]
we have
$$ \lim_{t\to \infty} \|\vp(t)\|=0 \quad \text{ for all } \vp \in \p.$$ 
Hence, $\inf\{ \|\vp(t)\| \,|\, t \geq 0\}=0$ for any $\vp \in \p$ with $\|\vp(0)\|=1$. The assertion then follows from Theorem~6.4 in \cite{stolyar95}.
\end{proof}


{\bf Proof of Theorem~\ref{thm:smooth-conv-lyap}}. In \cite[Theoerm~2]{questa-2012} it is shown that for the dynamical system under consideration there is a continuous Lyapunov pair if and only if the dynamical system is asymptotically stable. Thus, the non-converse implication is already shown.

Conversely, let the dynamical system $u$ defined on $\p$ be asymptotically stable. Then, by Theorem~2 in  \cite{questa-2012} there is a continuous Lyapunov pair $(V,W)$ such that
\begin{equation}\label{V-decrease}
V(\varphi(t)) - V(\varphi(s)) \leq -\,\int_{s}^{t} W(\varphi(r))\,\di{r} \quad \text{ for all } \vp \in \p, \, 0 \leq s\leq t.
\end{equation}
Thus, the construction of a $C^{\infty}$-smooth Lyapunov-pair remains.

To get differentiability on the boundary of the positive orthant, we first extend the pair $(V,W)$ to $\R^n$.  To this end, let $\abs \cdot$ denote the map that takes componentwise absolute values defined by  $ \abs x := (|x_1|,...,|x_n|)^{\trans} \in \R_+^n$. The extention of the pair $(V,W)$ to $\R^n$ is defined by
\begin{align*}
 V^e(x) := V(\abs x ), \qquad  W^e(x) := W( \abs x ).
\end{align*}
Note that, as a composition of continuous functions, the pair $(V^e,W^e)$ is also continuous. As a first consequence of assumption~(A) we conclude that $V^e$ is locally Lipschitz. 

\begin{lemma}\label{lem:V-loc-Lip}
Suppose the abstract dynamical system $u$ defined on $\p$ satisfies \emph{(A)} and is asymptotically stable. Then, $V^e$ is locally Lipschitz on $\R^n$. 
\end{lemma}
\begin{proof}
Let $U \subset \R^n$ be open, convex, and bounded and let $x \in U$. Following Corollary~3.7 in \cite{clarke1993subgradient}, since $-V^e$ is lower semicontinuous, it suffices to show that there is a $M>0$ such that for any $v \in \R^n$ it holds
\begin{align*}
D(-V^e)(x;v)\leq M \|v\|.
\end{align*}
Let $v'\in \R^n$ and $\xi>0$. Let $\varphi \in \p$ be a trajectory of the dynamical system satisfying  $\varphi(0)= \abs{x}$ and
\begin{equation*}
V^e(x) =  \int_0^{\infty} \|{\varphi}(s)\| \di{s}=  \int_0^{\| x\| \tau} \|{\varphi}(s)\| \di{s},
\end{equation*}
where in the last equality used the stability of $\p$ and Lemma~\ref{prop:stab}. By the continuity of $\abs \cdot$ we have
\begin{align*}
 \lim_{\xi \to 0} \abs{x+\xi v'} = \abs{x}.
\end{align*}
So, for every $\e>0$ and $\xi$ sufficiently small we have $ \abs{x+\xi v'} \in B(\abs{x}, \e) \cap \R_+^n$. Moreover, there is a continuous mapping $g:\R^n \to \R^n$ satisfying $\|g(v')\|=\|v'\|$ and 
\begin{equation*}
\abs{x+\xi v'} = \abs{x}  +\xi g(v').
\end{equation*}
For $T:= \max \{\|x\| \, \tau, \|x+\xi v'\| \,\tau \}< \tau (\|x\| + \e)$, by assumption~(A) and the triangular inequality, there are $c>0$ and  $\psi \in \p$ with $\psi(0)= \abs {x} +\xi g(v')$  such that 
\begin{align}\label{eq:varphi-R-estimate}
\|\, \varphi(t)\, \| - \|\,\psi(t)\,\| \leq \|\, \varphi(t)\,  - \,\psi(t)\,\| \leq \xi \, \|g(v') \|\, ( 1+ c (t) ) = \xi \, \|v' \|\, ( 1+ c (t) )
\end{align}
for all $t \in [0,T]$.
The definition of $V^e$, the stability of $\p$ together with Lemma~\ref{prop:stab}, and $\|\,\abs {x} +\xi g(v')\| = \|x +\xi v'\| $  yield
\begin{align*}
V^e(x+\xi v') \geq  \int_0^{\infty} \|\,{\psi}(s)\,\| \di{s} =  \int_0^{\|x + \xi v'\| \tau} \|\,{\psi}(s )\,\| \di{s}.
\end{align*}
On the one hand, if $\|x\| \leq \|x+\xi v'\|$ by using \eqref{eq:varphi-R-estimate} it follows
\begin{multline*} 
 V^e(x) - V^e(x+\xi v') \leq  \int_0^{\|x\|\tau}\|\,{\varphi}(s )\,\| \di{s}
- \int_0^{\|x+ \xi v'\|\tau} \|\,{\psi}(s )\,\| \di{s} \\
\leq  \int_0^{\|x\| \tau}\|\,{\varphi}(s )\| - \|{\psi}(s    )\,\| \di{s}\\
\leq \int_0^{\|x\|\tau}  \xi \,\|v'\|\,(1 +c(s) ) \di{s}
 \leq \xi \,\|v'\| \cdot\|x\|\,\tau\, C,
\end{multline*}
where $C:=\max\limits_{0 \leq s \leq \|x\|\tau}(1+c(s) )$. On the other hand, to consider the case $\|x\| > \|x+\xi v'\|$ we note that the triangle inequality together with the Lipschitz condition imply 
\begin{align}\label{lip}
\| \varphi(t)\| \leq \| \varphi(0 )\| \, + Lt \leq \|x\| \,(1 + L\tau) \quad \text{ for all }t \in [0,\|x\|\tau].
\end{align}
Using \eqref{eq:varphi-R-estimate}, \eqref{lip}, and $0 \leq \|x\| - \|x +\xi v'\| \leq \xi \|v'\|$ we obtain
\begin{multline*} 
V^e(x) - V^e(x+\xi v') \leq  \int_0^{\|x\|\tau}\|\,{\varphi}(s )\,\| \di{s} 
 - \int_0^{\|x+ \xi v'\|\tau} \|\,{\psi}(s )\,\| \di{s} \\
\leq  \int_0^{\|x+\xi v'\| \tau}\|\,{\varphi}(s )\| - \| {\psi}(s )\,\| \di{s} 
 + \int_{\|x +\xi  v'\|\tau}^{\|x\|\tau}\|\,{\varphi}(s )\,\|\di{s}\\
\leq \int_0^{\|x+\xi v'\|\tau}  \xi \,\|v'\|\,(1 +c(s)) \di{s} 
 + \tau (\, \|x\| - \|x +\xi v'\|\,) \,\cdot  \sup_{s \in [\|x +\xi v'\|\tau, \|x\|\tau ] } \|\,{\varphi}(s )\,\|\\
\leq \xi \,\|v'\| \,\|x+\xi v'\|\,\tau\,C \,+ \,\tau \,\xi \,\|v'\| \,\|x\| \,( 1+L\tau).
\end{multline*}
Consequently, taking limits and using that $U$ is bounded there is a $M>0$ such that
\begin{align*}
D(-V^e)(x;v)= \liminf_{\xi \to 0, v'\rightarrow v} \frac{V^e(x) - V^e(x+\xi v') }{\xi}  \leq \tau\left(C+1+L\tau\right) \, \|x\| \cdot \|v\| \leq M \, \|v\|.
\end{align*}
The shows the assertion.
\end{proof}

\medskip

Proceeding with the construction of a smooth Lyapunov pair, let $U$ be an open subset of $\R^n$ and consider the convolution of $V^e$ and $k_{r}$ defined by
\begin{equation*}
V^e_{r}(x) := V^e \ast k_{r} \, (x) = \int_{\R^n} V^e(x-y)\, k_{r}(y)\, \di{y}= \int_{\R^n} V(\abs {x-y})\, k_r(y)\, \di{y}.
\end{equation*}
Also, we consider the convolution of $W^e$ and $k_r$ given by
\begin{equation*}
 W^e_{r}(x) := W^e \ast k_{r} \, (x) = \int_{\R^n} W^e(x-y)\, k_{r}(y)\, \di{y}.
\end{equation*}
By standard convolution results it follows $V^e_{r} \in C^{\infty}(U,\R_+)$ and $W^e_{r} \in C^{\infty}(U,\R_+)$. Furthermore, since $V^e$ is continuous on $U$ it holds $V^e_r \rightarrow V^e$ uniform on compact subsets of $U$ as $r\rightarrow 0$. Consequently, for every $\e >0$ there is an $r_0$ such that for all $r \in (0,r_0)$ we have $V^e_r$ and $W^e_r$ are smooth on $U$ and 
\begin{align}\label{cond:conv_uoc}
| V^e_r(x) -V^e(x)| \leq \e, \qquad |W^e_r(x) - W^e(x)| \leq \tfrac{\e}{2}  \qquad \text{ for all } x \in U.
\end{align} 
The subsequent statement addresses the decrease condition of the convolution along trajectories $\vp \in \p$.

\begin{lemma}\label{lem:local-conv}
Let  $U \subset \R^n$ be compact such that $U\cap \R_+^n \neq \emptyset$ and suppose $(V,W)$ satisfy \eqref{V-decrease} and assumption~\emph{(A)} is satisfied. Then, for every $\e >0$ there exists a $r_0 >0$ such that for all $r\in(0,r_0)$ we have
\begin{align}\label{est:V-conv-local}
\dot V^e_{r}(\varphi(t)) \leq - W^e(\varphi(t)) + \e
\end{align}
for all $\varphi \in \p$ and $ t\in [0,T]$ with $\varphi(\cdot)|_{[0,T]} \subset U\cap \R_+^n$.
\end{lemma}
\begin{proof}
Let $\varphi \in \p$ be a trajectory satisfying $\varphi(0)=x \in U\cap \R_+^n$. Then, for $h>0$ we have
\begin{align*}
V^e_{r}\big(\varphi(t+h )\big) - V^e_{r}\big(\varphi(t )\big) 
= \int_{\R^n} \Big(\, V^e\big( \varphi(t+h ) -y\big) - V^e\big(\varphi(t )-y\big)\,\Big) \, k_{r}(y)\, \di{y}.
\end{align*}
There is a continuous mapping $g:\R^n \to \R^n$ satisfying $\|g(y)\|=\|y\|$ and 
\begin{equation*}
\abs{\varphi(t) -y} =  \varphi(t)   -g(y).
\end{equation*}
Further, by assumption~(A) and for $h$ sufficiently small there is a continuous function $c\colon \R_+ \to \R_+$ with $\lim_{t\to 0}\tfrac{c(t)}{t}=:c_0>0$ and a trajectory $\psi(t+\cdot) \in \p$ with $\psi(t)=\varphi(t)-g(y)$ such that
\begin{equation}\label{estimate-h-small}
\|  \varphi(t+h) -y - \psi(t+h ) \| \leq  \|y\| c(h).
\end{equation}
Using this, as $V^e( \psi(t+h) ) = V(\psi(t+h))$ we obtain
\begin{multline}\label{Ve-decrease}
V^e_{r}(\varphi(t+h) - V^e_{r}(\varphi(t )
\leq\int_{\R^n} \Big|\, V^e(\varphi(t+h)-y)\,-\, V^e(\psi(t+h )  \,\Big| \, k_{r}(y)\, \di{y} \\
+ \int_{\R^n} \Big(\, V(\psi(t+h ) - V(\varphi (t) -g(y) )\,\Big)  \, k_{r}(y)\, \di{y}.
\end{multline}
By the local Lipschitz continuity of $V$ with constant $L$ and \eqref{estimate-h-small}, the first term on the right hand side in the above inequality can be estimated as follows
\begin{multline*}
\int_{\R^n} \Big|\, V^e(\varphi(t+h )-y)\,-\, V^e( \psi(t+h ) )  \,\Big|\, k_{r}(y)\, \di{y} \\
=\int_{\R^n} \Big|\, V(\varphi(t+h )-g(y))\,-\, V( \psi(t+h) )  \,\Big|\, k_{r}(y)\, \di{y} \\
\leq  \int_{\R^n} L \big\| \, \varphi(t+h )-g(y) \,-\,\psi(t+h ) \,\big\|\, k_{r}(y)\, \di{y} \\
\leq  c(h) \,L \int_{\R^n}  \|g(y)\| \cdot k_{r}(y)\, \di{y}= c(h)\,L \int_{\R^n}  \|y\| \cdot k_{r}(y)\, \di{y}.
\end{multline*}
Furthermore, it holds $\displaystyle \int_{B(0,r)} \|y\| \, k_r (y) \di{y} \leq \displaystyle \int_{B(0,r)}  r\, k_r(y) \di{y} = r$ and choosing $r_0 :=\frac{\e}{2\,c_0\, L}$ it follows
\begin{align*}
\int_{\R^n} \big| \,V^e(\varphi(t+h )-y)\,-\, V(\psi(t+h  ) )  \,\big|\cdot k_{r}(y)\, \di{y} 
\leq  \tfrac{c(h)}{c_0}\,\tfrac{\e}{2}.
\end{align*}
Asymptotic stability of the dynamical system implies that the very last term in \eqref{Ve-decrease} can be estimated by means of the function $W $ and its mollification,
\begin{multline*}
\int_{\R^n} \Big(\, V\big(\psi(t+h )\big) - V\big(\varphi (t ) -g(y) \big)\,\Big)  \, k_{r}(y)\, \di{y}\\
\leq \int_{\R^n} \left(  - \int_t^{t+h} W \big(\psi(s )\big) \di{s} \right) \cdot k_{r}(y) \,\di{y} \\
= - \int_0^{h} \left( \,\int_{\R^n} W \big(\psi(t+s )\big) \cdot k_{r}(y) \,\di{y}  \right) \di{s},
\end{multline*}
where the last identity is obtained by integration by substitution. Next, we show that the function
\begin{align*}
 s \mapsto \int_{\R^n} W \big(\psi(t+s  ) \big) \, k_{r}(y) \,\di{y}
\end{align*}
is continuous in $[0,h]$. To see this, consider the \emph{modulus of continuity} of the function 
$$s\mapsto W\big(\psi(t+s)\big),$$
defined for $\delta \in [0,h]$ by
\begin{align*}
\mbox{m}\Big(\delta,W\big( \psi(t+\cdot  )  \big)\Big):=
\sup_{|s-s'|\leq \delta} \Big|W\big( \psi(t+s  ) \big) - W\big(\psi(t+s' )\,\big) \Big|.
\end{align*}
Then, for $s,s' \in [0,h]$ it holds 
\begin{align*}
W \big( \psi(t+s  ) )\,\big) - W \big( \psi(t+s'  )   \,\big) \leq   \mbox{m}\Big(h,W\big( \psi(t+ \cdot  ) \,\big) \Big).
\end{align*}
By asymptotic stability of the dynamical system $\| \psi(t+s ) \|$ is bounded and, hence, $W(\psi(t+\cdot  )\,)$ is uniformly continuous. Thus, we have
\begin{align*}
\lim_{h\rightarrow 0}\, \mbox{m}\, \Big(h,W\big( \psi(t+\cdot  )     \big)\Big) = 0.
\end{align*}
That is, for every $\e'>0$ there is a $\delta'>0$ such that $\mbox{m}\Big(h,W \big(\,\psi(t+\cdot )\,\big) \Big) \leq \e'$  for all  $h \leq \delta'$. For $\e' >0$ choose $\delta>0$ such that $|s- s'|< \delta <\delta'$. Then, 
\begin{multline*}
\int_{\R^n} \Big( W \big(  \psi(t+s  )   \big) - W \big(  \psi(t+s'  )    \big) \Big)\, k_{r}(y) \,\di{y} \\
\leq 
\int_{\R^n} \mbox{m}  \Big( \delta, W\big(   \psi(t+\cdot  )    \big) \Big) \, k_{r}(y) \,\di{y} 
\leq 
\int_{\R^n} \e' \,  k_{r}(y) \,\di{y} =\e'.
\end{multline*}

Moreover, by conditions \eqref{cond:conv_uoc} we have $-W^e_r(x) + \frac{\e}{2} \leq -W^e(x) + \e$. Finally, the collection of the above relations yields
\begin{align*}
\dot V_r^e(\varphi(t)) &= \lim_{h \rightarrow 0}\frac{V^e_{r}\big(\varphi(t+h)\big) - V_{r}^e\big(\varphi(t)\big)}{h} \\&
\leq \frac{\e}{2}  - \lim_{h \rightarrow 0} \frac{1}{h}  \int_{\R^n}  \left(  \int_0^h W\big(  \psi(t+s) \big) \, k_{r}(y)\, dy\right) \di{s}
\\ 
& \leq 
 - \lim_{h \rightarrow 0}
\frac{1}{h} \int_0^h \left( \int_{\R^n}   W\big(  \psi(t+s) \big) \, k_{r}(y)\, dy\right) \di{s} +\frac{\e}{2}\\ &
=   - \int_{\R^n} W\big(  \psi(t) \big) \, k_{r}(y)\, \di{y} + \frac{\e}{2} 
=   - \int_{\R^n} W\big(  \varphi(t)-g(y) \big) \, k_{r}(y)\, \di{y} + \frac{\e}{2} \\&
=   - \int_{\R^n} W^e( \varphi(t)-y ) \, k_{r}(y)\, \di{y} + \frac{\e}{2} \\
&= - W^e_r(\varphi(t)) + \frac{\e}{2} \leq   - W^e(\varphi(t)) + \e.
\end{align*}
This shows Lemma~\ref{lem:local-conv}.
\end{proof}

\medskip

Now, let $\mathcal{U} = \{ U_{i}\}_{i=1}^{\infty}$ be a locally finite open cover of $\R^n$ such that for every $i$ the closure $\overline{U_i}$ is compact. Further, let $\{\psi_i\}_{i=1}^{\infty}$ be a smooth partition of unity that is subordinate to $\mathcal{U}$. Define
\begin{align}\label{choice:eps-i}
\e_i = \tfrac{1}{4}\min\{ \min_{x \in \bar U_i} V^e(x),\min_{x \in \bar U_i} w^e(x)\} \qquad \mbox{and} \qquad q_i = \max_{x \in \bar U_i} \|\nabla \psi_{i}(x)\|.
\end{align}
Then, by Lemma~\ref{lem:local-conv} for every $i$ there is a $C^{\infty}$-pair $(V^e_i, W_i^e)$ such that for every $x \in U_i$,
\begin{align}\label{est:V-w-local}
|V^e(x) -  V^e_{i}(x)| <  \frac{\e_i}{2^{i+1}(1+ q_i)}  \qquad \mbox{and} \qquad  |W^e(x) -  W^e_{i}(x)| <  \e_i.
\end{align}
Moreover, by the conditions \eqref{est:V-conv-local} and \eqref{choice:eps-i} we have that
\begin{align}\label{est:V-dot-local}
\dot V^e_i\big(\varphi(t )\big) \leq - W^e\big(\varphi(t)\big) +  2\,\e_i \leq - \tfrac{1}{2} W^e\big(\varphi(t)\big) .
\end{align}
Next, we define
\begin{align*}
V^e_s(x) := \sum_{i=1}^{\infty} \psi_i(x)\,  V^e_i(x).
\end{align*}
The following estimate holds true
\begin{align*}
| V^e_{s}(x) - V^e(x)|  \leq  \sum_{i=1}^{\infty} \psi_i(x) \, \big|V^e_i(x) - V^e(x) \big|
\leq  \frac{V^e(x)}{4}\, \sum_{i=1}^{\infty}  \frac{\psi_i(x)}{2^{i+1}(1+q_i)} \leq \tfrac{1}{8} V^e(x).
\end{align*}
Using the triangular inequality, the latter estimate shows that $V^e_s$ is proper and positive definite. The next step is to derive that $V^e_s$ is decaying along trajectories of $\p$. To this end, we consider 
\begin{multline*}
\tfrac{\di{}}{\di{t}}[V^e_s(\varphi(t))] = \tfrac{\di{}}{\di{t}}\left[ V^e(\varphi(t))+ V^e_s(\varphi(t)) -  V^e(\varphi(t))\right] \\
= \tfrac{\di{}}{\di{t}}[ V^e(\varphi(t))] +  \tfrac{\di{}}{\di{t}}\left[ \sum_{i=1}^{\infty} \psi_i(\varphi(t))\, \Big(V^e_i(\varphi(t)) -  V^e(\varphi(t)) \Big) \right] \\
= \dot V^e(\varphi(t))  + \sum_{i=1}^{\infty}  \psi_i(\varphi(t)) \,\, \Big( \dot V^e_i(\varphi(t)) - \dot V^e(\varphi(t)) \Big) 
  + \sum_{i=1}^{\infty} \dot \psi_i(\varphi(t))\,\, \Big( V^e_i(\varphi(t)) - V^e(\varphi(t))\Big)\\
\leq \sum_{i=1}^{\infty}  \psi_i(\varphi(t)) \,\left( \dot V^e_i(\varphi(t))  + \sum_{j=1}^{\infty} \dot \psi_j(\varphi(t))\,\Big| V^e_j(\varphi(t)) - V^e(\varphi(t))\Big|\right).
\end{multline*}
Using the conditions \eqref{est:V-w-local} and \eqref{est:V-dot-local} we get the following estimate
\begin{align*}
 \dot V^e_s(\varphi(t)) \leq \sum_{i=1}^{\infty}  \psi_i(\varphi(t)) \,\left( - \tfrac{1}{2}\,W^e(\varphi(t)) + \sum_{j=1}^{\infty} \frac{q_j \e_j}{2^{j+1}(1+q_j)}\right).
\end{align*}
Defining $\widetilde \e_i:= \max \{ \e_j \, : \, x \in U_i \cap U_j \not = \emptyset \}$ we have that
\begin{align*}
\dot V^e_s(\varphi(t)) &\leq \sum_{i=1}^{\infty}  \psi_i(\varphi(t)) \,\left( - \tfrac{1}{2}\,W^e(\varphi(t)) + \widetilde \e_i \sum_{j=1}^{\infty} \tfrac{1}{2^{j+1}}\right)\\
&= \sum_{i=1}^{\infty}  \psi_i(\varphi(t)) \,\big( - \tfrac{1}{2} W^e(\varphi(t)) +  \widetilde \e_i\big).
\end{align*}
Using \eqref{choice:eps-i} and the triangular inequality applied to the second inequality in \eqref{est:V-w-local}, it holds that 
\begin{align*}
- \tfrac{1}{2} W^e(\varphi(t)) +  \widetilde \e_i \leq - \tfrac{1}{4} W^e(\varphi(t)) \leq - \tfrac{1}{5} W_i^e(\varphi(t)).
\end{align*}
Finally, we have that
\begin{align*}
\dot V^e_s(\varphi(t))\leq  - \tfrac{1}{5}\sum_{i=1}^{\infty}  \psi_i(\varphi(t))  \, W^e_i(\varphi(t)) =: - W^e_s(\varphi(t)).
\end{align*}

Consequently, the pair $(V^e_s,W^e_s)$ defines a $C^\infty$-smooth Lyapunov pair, which shows the assertion.\hfill $\Box$

\section*{Acknowledgment}\label{sec:acknowledgment}

I am grateful to Fabian Wirth for many fruitful discussions and very helpful comments.

%

\end{document}